\documentclass[12pt,a4paper]{article}
\usepackage{amssymb,amsmath,amsthm}
\usepackage[english]{babel}
\usepackage{t1enc}
\usepackage[latin2]{inputenc}
\usepackage{epsfig}
\usepackage{subfig}
\usepackage{comment}
\usepackage{epstopdf}
\usepackage{url}
\usepackage{hyperref}
\usepackage{cleveref}
\usepackage{enumitem}

\usepackage[margin=1.1in]{geometry}

\newtheorem{thm}{Theorem}
\newtheorem{manualtheoreminner}{Theorem}
\newenvironment{thm'}[1]{%
	\renewcommand\themanualtheoreminner{#1}%
	\manualtheoreminner
}{\endmanualtheoreminner}
\newtheorem{cor}[thm]{Corollary}

\newtheorem{lem}[thm]{Lemma}
\newtheorem{prop}[thm]{Proposition}

\theoremstyle{remark}

\usepackage{indentfirst}
\def\cA{\mathcal{A}}
\def\cH{\mathcal{H}}
\def\cI{\mathcal{I}}
\def\cF{\mathcal{F}}

\def\cD{\mathcal{D}}
\def\cC{\mathcal{C}}
\def\cS{\mathcal{S}}
\def\cL{\mathcal{L}}
\def\cE{\mathcal{E}}
\def\cV{\mathcal{V}}
\def\R{\mathbb{R}}

\makeatletter
\def\blfootnote{\gdef\@thefnmark{}\@footnotetext}
\makeatother

\begin{document}

\title{On tangencies among planar curves with an application to coloring L-shapes}
\author{
Eyal Ackerman\thanks{Department of Mathematics, Physics and Computer Science,
	University of Haifa at Oranim, 	Tivon 36006, Israel.}\and
Bal\'azs Keszegh\thanks{Alfr\'ed R{\'e}nyi Institute of Mathematics, Hungarian Academy of Sciences and MTA-ELTE Lend\"ulet Combinatorial Geometry Research Group, Institute of Mathematics, E\"otv\"os Lor\'and University, Budapest, Hungary. Research supported by the Lend\"ulet program of the Hungarian Academy of Sciences (MTA), under the grant LP2017-19/2017 and by the National Research, Development and Innovation Office -- NKFIH under the grant K 132696.}
\and D\"om\"ot\"or P\'alv\"olgyi\thanks{MTA-ELTE Lend\"ulet Combinatorial Geometry Research Group, Institute of Mathematics, E\"otv\"os Lor\'and University, Budapest, Hungary. Research supported by the Lend\"ulet program of the Hungarian Academy of Sciences (MTA), under the grant LP2017-19/2017.}
}
\maketitle

\begin{abstract}
We prove that there are $O(n)$ tangencies among any set of $n$ red and blue planar curves in which every pair of curves intersects at most once and no two curves of the same color intersect.
If every pair of curves may intersect more than once, then it is known that the number of tangencies could be super-linear.
However, we show that a  linear upper bound still holds if we replace tangencies by pairwise disjoint connecting curves that all intersect a certain face of the arrangement of red and blue curves.

The latter result has an application for the following problem studied by 
Keller, Rok and Smorodinsky [Disc.\ Comput.\ Geom.\ (2020)] in the context of \emph{conflict-free coloring} of \emph{string graphs}: what is the minimum number of colors that is always sufficient to color the members of any family of $n$ \emph{grounded L-shapes} such that among the L-shapes intersected by any L-shape there is one with a unique color? They showed that $O(\log^3 n)$ colors are always sufficient and that $\Omega(\log n)$ colors are sometimes necessary.
We improve their upper bound to $O(\log^2 n)$. 
\end{abstract}

\section{Introduction}

The \emph{intersection graph} of a collection of geometric shapes is the graph whose vertex set consists of the shapes, and whose edge set consists of pairs of shapes with a non-empty intersection.
Various aspects of such graphs have been studied vastly over the years.
For example, by a celebrated result of Koebe~\cite{Koebe36} planar graphs are exactly the intersection graphs of interior-disjoint disks in the plane.
Another example which is more recent and more related to our topic is a result
by Pawlik et al.~\cite{chibounded}, who showed that intersection graphs of planar segments are not \emph{$\chi$-bounded}, that is, their chromatic number cannot be upper-bounded by a function of their clique number.

We will mainly consider (not necessarily closed) planar curves (Jordan arcs).
A family of curves is \emph{$t$-intersecting} if every pair of curves intersects in at most $t$ points.
We say that two curves \emph{touch} each other at a \emph{touching} (\emph{tangency}) point $p$ if both of them contain $p$ in their interior, $p$ is their only intersection point and it is not a crossing point.\footnote{$p$ is a \emph{crossing point} of two curves if there is a small disk $D$ centered at $p$ which contains no other intersection point of these curves, each curve intersects the boundary of $D$ at exactly two points and in the cyclic order of these four points no two consecutive points belong to the same curve.} 

According to a nice conjecture of Pach \cite{pachpc} the number of tangencies among a $1$-intersecting family $\cS$ of $n$ curves should be $O(n)$ if every pair of curves intersects. 
Gy{\"o}rgyi et al.~\cite{GYORGYI201829} proved this conjecture in the special case where there are constantly many faces of the arrangement of $\cS$ such that every curve in $\cS$ has one of its endpoints inside one of these faces.
Here we prove the following variant.

\begin{thm}\label{thm:touchings}
	Let $\cS$ be a $1$-intersecting set of red and blue curves such that no two curves of the same color intersect.
	Then the number of tangencies among the curves in $\cS$ is $O(n)$.
\end{thm}

Note that it is trivial to construct an example with $\Omega(n)$ tangencies.
Theorem~\ref{thm:touchings} does not hold if a pair of curves in $\cS$ may intersect twice.
Indeed, Pach et al.~\cite{Treml} considered the following problem: what is the maximum number $f(n)$ of tangencies among $n$ $x$-monotone red and blue curves where no two curves of the same color intersect.
They showed that $\Omega(n\log n) \le f(n) \le O(n\log^2 n)$, where
their lower bound construction is $2$-intersecting (but not $1$-intersecting).

The number of tangencies within a $1$-intersecting set of $n$ ($x$-monotone) curves can be $\Omega(n^{4/3})$: it is not hard to obtain this bound using the famous construction of Erd\H{o}s of $n$ points and $n$ lines that determine that many point-line incidences \cite{pachbook}.
An almost matching upper bound of $O(n^{4/3}\log^{2/3}n)$ follows from a result of Pach and Sharir~\cite{PS91}.\footnote{A more careful analysis yields the upper bound $O(n^{4/3}\log^{1/3}n)$.}

\paragraph{Connecting curves.}
Instead of considering touching points among curves in a family of curves $\cS$, we may consider pairs of disjoint curves that are intersected by a curve $c$ from a different family of curves $\cC$, such that $c$ does not intersect any other curve.
Indeed, each touching point of two curves from $\cS$ can be replaced by a new, short curve that connects the two previously touching curves that become disjoint by redrawing one of them near the touching point.
Conversely, if $\cC$ consists of curves that connect \emph{disjoint} curves from $\cS$, then each connecting curve in $\cC$ can be replaced by a touching point between the corresponding curves by redrawing one of these two curves.
Therefore, studying touching points or such connecting curves are equivalent problems.
This gives the following reformulation of Theorem \ref{thm:touchings}.

\begin{thm'}{\ref*{thm:touchings}'}\label{lem:non-neighboring}\label{thm:non-neighboring}
	Let $\cS$ be a set of $1$-intersecting $n$ red and blue curves such that no two curves of the same color intersect.
	Suppose that $\cC$ is a set of pairwise disjoint curves such that each of them intersects exactly a distinct pair of disjoint curves from $\cS$. 
	Then $|\cC| = O(n)$.
\end{thm'}

If $\cS$ and $\cC$ are two families of curves, then we say that $\cC$ is \emph{grounded} with respect to $\cS$ if there is a connected region of $\R^2\setminus \cS$ that contains at least one point of every curve in $\cC$.
If $\cC$ is grounded with respect to $\cS$, then we can drop the assumption that $\cS$ is $1$-intersecting and prove the following variant.

\begin{thm}\label{thm:groundedcurves}
	Let $\cS$ be a set of $n$ red and blue curves such that no two curves of the same color intersect.
	Suppose that $\cC$ is a set of pairwise disjoint curves grounded with respect to $\cS$, such that each of them intersects exactly a distinct pair of curves from $\cS$. 
	Then $|\cC| = O(n)$.
\end{thm}

Note that we have also dropped the assumption that a red curve and a blue curve which are connected by a curve in $\cC$ are disjoint. Therefore it is essential that $\cC$ is grounded with respect to $\cS$, for otherwise we might have $|\cC| =\Omega(n^2)$.
Indeed, let $\cS$ be a ($1$-intersecting) set of $n/2$ horizontal segments and $n/2$ vertical segments such that every horizontal segment and every vertical segment intersect. Then each such pair can be connected by a curve in $\cC$ very close to their intersection point. Hence $|\cC|=n^2/4$.

\medskip
Clearly, instead of \emph{curves}, Theorem \ref{thm:groundedcurves} could also be stated with a red and a blue family of disjoint \emph{shapes}, with no requirement at all about the shapes, except that each of them is connected.

\medskip
A corollary of Theorem \ref{thm:groundedcurves} improves a result of Keller, Rok and Smorodinsky~\cite{keller2020conflictfree} about conflict-free colorings of $L$-shapes.

\paragraph{Coloring L-shapes.}
An \emph{L-shape} consists of a vertical line-segment and a horizontal line-segment such that the left endpoint of the horizontal segment coincides with the bottom endpoint of the vertical segment (as in the letter {\sf 'L'}, hence the name). Whenever we consider a family L-shapes, we always assume that no pair of them have overlapping segments, that is, they have at most one intersection point. 

A family of L-shapes is \emph{grounded} if there is a horizontal line that contains the top point of each L-shape in the family.
A (grounded) \emph{L-graph} is a graph that can be represented as the intersection graph of a family of (grounded) L-shapes.
Gon{\c{c}}alves et al.~\cite{GoncalvesIP18} proved that every planar graph is an L-graph.
The line graph of every planar graph is also known to be an L-graph~\cite{FELSNER201648}.
As for grounded L-graphs, McGuinness~\cite{MCGUINNESS1996179}
proved that they are $\chi$-bounded.
Jel{\'{\i}}nek and T{\"{o}}pfer~\cite{Jelinek019} characterized grounded L-graphs in terms of vertex ordering with forbidden patterns.

Keller, Rok and Smorodinsky~\cite{keller2020conflictfree} studied \emph{conflict-free} colorings of string graphs\footnote{A \emph{string graph} is the intersection graph of curves in the plane.} and in particular of grounded L-graphs.
A coloring of the vertices of a hypergraph is \emph{conflict-free}, if every hyperedge contains a vertex whose color is not assigned to any of the other vertices of the hyperedge.
The minimum number of colors in a conflict-free coloring of a hypergraph $\cH$ is denoted by $\chi_{\rm CF}(\cH)$.
There is a vast literature on conflict-free coloring of hypergraphs that stem from geometric settings due to its application to frequencies assignment in wireless networks and its connection to \emph{cover-decomposability} and other coloring problems---see the survey of Smorodinsky~\cite{surveycf}, the webpage~\cite{cogezoo} and the references within.

One natural way of defining a hypergraph $\cH (\cS)$ with respect to a family of geometric shapes $\cS$ is as follows: the vertex set is $\cS$ and for every shape $S \in \cS$ there is a hyperedge that consists of all the members of $\cS \setminus \{S\}$ whose intersection with $S$ is non-empty.\footnote{If two shapes give rise to the same hyperedge, then this hyperedge appears only once in the hypergraph.}
This is the so-called \emph{punctured} or \emph{open neighborhood} hypergraph of the intersection graph.
Similarly, for two families of shapes, $\cS$ and $\cF$, the hypergraph $\cH(\cS,\cF)$ has $\cS$ as its vertex set and has a hyperedge for every $F \in \cF$ which consists of all the members of $\cS \setminus \{F\}$ whose intersection with $F$ is non-empty. Hence, $\cH(\cS)=\cH(\cS,\cS)$. 

\medskip
Using Theorem~\ref{thm:groundedcurves} we can improve the following result.
 
\begin{thm}[Keller, Rok and Smorodinsky~\cite{keller2020conflictfree}]
\label{thm:keller}
$\chi_{\rm CF}({\cH(\cL)}) = O(\log^3 n)$ for every set $\cL$ of $n$ grounded L-shapes.
Furthermore, for every $n$ there exists a set $\cL$ of $n$ grounded L-shapes such that $\chi_{\rm CF}({\cH(\cL)}) = \Omega(\log n)$.
\end{thm}

In order to obtain this result and many other results considering (conflict-free) coloring of hypergraphs it is often enough to consider the chromatic number of a sub-hypergraph consisting of hyperedges of size two, that is, the \emph{Delaunay graph}.
For two families of geometric shapes $\cS$ and $\cF$, 
the \emph{Delaunay graph} of $\cS$ and $\cF$, denoted by $\cal D(S,F)$, is the graph whose vertex set is $\cS$ and whose edge set consists of pairs of vertices such that there is a member of $\cF$ that intersects exactly the shapes that correspond to these two vertices and no other shape.
Note that if $\cS$ is a set of planar points and $\cF$ is the family of all disks, then $\cal D(S,F)$ is the standard Delaunay graph of the point set $\cS$.

A key ingredient in the proof of Theorem~\ref{thm:keller} in~\cite{keller2020conflictfree} is the following lemma.

\begin{lem}[{\cite[Proposition 3.9]{keller2020conflictfree}}] \label{lem:keller}
Let $\cL \cup \cI$ be a set of grounded L-shapes such that $|\cL|=n$ and the L-shapes in $\cI$ are pairwise disjoint.
Then $\cD(\cL,\cI)$ has $O(n \log n)$ edges.
\end{lem}

Theorem~\ref{thm:groundedcurves} clearly implies Lemma~\ref{lem:keller}, since every family of $n$ L-shapes consists of $n$ 
pairwise disjoint horizontal (red) segments and $n$ pairwise disjoint vertical (blue) segments. 
Thus, Theorem~\ref{thm:groundedcurves} is a twofold improvement of Lemma~\ref{lem:keller}: we consider a more general setting and prove a better upper bound.
Furthermore, our proof is simpler than the proof of Lemma~\ref{lem:keller} in~\cite{keller2020conflictfree} (especially the weak version of Theorem~\ref{thm:groundedcurves} that we prove in Section~\ref{sec:linear}).

\medskip
By replacing Lemma~\ref{lem:keller} with Theorem~\ref{thm:groundedcurves} (or its weaker version) in the proof of Theorem~\ref{thm:keller} we obtain a better upper bound for the number of colors that suffice to conflict-free color $n$ grounded L-shapes.

\begin{thm} \label{thm:cf-grounded}
	Let $\cL$ be a set of $n$ grounded L-shapes.
	Then it is possible to color every L-shape in $\cL$ with one of $O(\log^2 n)$ colors such that
	for each $\ell \in \cL$ there is an L-shape with a unique color among the L-shapes whose intersection with $\ell$ is non-empty.  
\end{thm} 

The upper bound on the number of edges in the Delaunay graph also implies upper bounds for the number of hyperedges of size at most $k$, the chromatic number of the hypergraph and its VC-dimension.\footnote{Recall that the VC-dimension of a hypergraph $\cH=(\cV,\cE)$ is the size of its largest subset of vertices $\cV' \subseteq \cV$ that can be \emph{shattered}, that is, for every subset $\cV'' \subseteq \cV'$ there exists a hyperedge $h \in \cE$ such that $h \cap \cV' = \cV''$.} 
This was already shown, e.g., for \emph{pseudo-disks}~\cite{aronov21,buzaglo}, however, the same arguments apply in general.
We summarize these facts in the following statement.

\begin{thm}\label{thm:2unif-kunif}	
	Let $\cH=(\cV,\cE)$ be an $n$-vertex hypergraph. 
	Suppose that there exist absolute constants $c, c' \ge 0$ such that for every $\cV' \subseteq \cV$ the Delaunay graph of the sub-hypergraph\footnote{The hyperedges of this sub-hypergraph are the non-empty subsets in $\{h \cap \cV' \mid h \in \cE\}$; this is sometimes called the \emph{trace}, or the restriction to $\cV'$.} induced by $\cV'$ has at most $c|\cV'|-c'$ edges, then:
	\begin{enumerate}[label=(\roman*)]
		\item The chromatic number of $\cH$ is at most $2c+1$ (at most $2c$ if $c' > 0$);
		\item The VC-dimension of $\cH$ is at most $2c+1$ (at most $2c$ if $c'>0$); and
		\item $\cH$ has $O(k^{d-1}n)$ hyperedges of size at most $k$ where $d$ is the VC-dimension of $\cH$.
	\end{enumerate}
\end{thm}

Using a result of Chan et al.~\cite{hittingset} this has another consequence about finding hitting sets. We follow the statement and the usage of this theorem as in~\cite{aronov21} (see therein the definition of the minimum weight hitting set problem).

\begin{thm}\cite{hittingset}\label{thm:hittingset}
	Let $\cH=(\cV,\cE)$ be a hypergraph, where the number of edges of cardinality $k$ for any restriction of $\cH$ to a subset $\cV' \subset \cV$ is at most $O(|\cV'|k^c)$, where $c > 0$ is an absolute constant and $k\le |\cV'|$ is an integer	parameter. Then there exists a randomized polynomial-time $O(1)$-approximation algorithm for the minimum weight hitting set problem for $\cH$.
\end{thm}

Theorem \ref{thm:2unif-kunif} and Theorem \ref{thm:hittingset} together imply:

\begin{thm}\label{thm:hittingsetfromdel}	
	Let $\cH=(\cV,\cE)$ be an $n$-vertex hypergraph. 
	Suppose that there exist absolute constants $c, c' \ge 0$ such that for every $\cV' \subseteq \cV$ the Delaunay graph of the sub-hypergraph induced by $\cV'$ has at most $c|\cV'|-c'$ edges, then there exists a randomized polynomial-time $O(1)$-approximation algorithm for the minimum weight hitting set problem for $\cH$.
\end{thm}

Thus, we have:

\begin{cor}\label{cor:curves-groundedcurves}
	Let $\cS$ be a set of $n$ red and blue curves, such that no two curves of the same color intersect and let $\cC$ be another set of pairwise disjoint curves which is grounded with respect to $\cS$.
	Then the chromatic number of the intersection hypergraph $\cH(\cS,\cC)$
	and its VC-dimension are bounded by a constant, and for every $k$ the number of hyperedges of size at most $k$ in $\cH(\cS,\cC)$ is $k^{O(1)}n$. Also there exists a randomized polynomial-time $O(1)$-approximation algorithm for the minimum weight hitting set problem for $\cH$.
\end{cor}

Note that the upper bounds on the chromatic number and VC-dimension can be deduced easily in this case without using Theorem~\ref{thm:2unif-kunif} (see Section~\ref{sec:general-bounds}).

\paragraph{Organization.}
In Section~\ref{sec:1'} we prove Theorem~\ref{thm:touchings}.
In Section~\ref{sec:linear} we prove a weaker version of Theorem~\ref{thm:groundedcurves} which is still stronger than  Lemma~\ref{lem:keller} and suffices for proving Theorem~\ref{thm:cf-grounded}.  
Then in Section~\ref{sec:cf} we give an outline of the proof of Theorem~\ref{thm:keller} in~\cite{keller2020conflictfree} and explain how by plugging into it Theorem~\ref{thm:groundedcurves} or its weaker version we obtain Theorem~\ref{thm:cf-grounded}.
Theorem~\ref{thm:groundedcurves} is proved in Section~\ref{sec:groundedcurves}. 
For completeness, we prove Theorem~\ref{thm:2unif-kunif} in Section~\ref{sec:general-bounds}.

In order to avoid technicalities and pathological cases we assume henceforth that all the curves that we consider are non-self-intersecting polygonal chains consisting of finitely many segments and that every pair of curves intersect at finitely many points. Each intersection point involves exactly two curves and is either a proper crossing of these curves or an endpoint of one of them that belongs to the interior of the other curve.\footnote{For L-shapes these properties are trivially satisfied after an appropriate small perturbation.} 

\section{Proof of Theorem~\ref{thm:touchings}}\label{sec:1'}
We begin with the following definitions which we will use throughout the paper.
Let $\cC$ be a set of curves in the plane.
Then $\cC$ induces a partition of the plane, which is called the \emph{arrangement} of $\cC$,
$\cA_\cC$, and consists of \emph{vertices}, \emph{edges}\footnote{Not to be confused with vertices and edges of a graph.} and \emph{faces}.
A \emph{vertex} is either an endpoint of a curve or an intersection point of (two) curves;
an \emph{edge} is a maximal sub-curve that does not contain any vertices in its interior; and
a \emph{face} is the closure of a maximal connected region of $\mathbb{R}^2\setminus \cC$.
The vertices and the edges of $\cA_\cC$ naturally induce a plane graph $G_{\cC}$. 
The size of a face $F$, denoted by $|F|$, is the number of edges that are adjacent to $F$, where the cut-edges (bridges) in $G_{\cC}$ are counted with multiplicity two. 
We denote by $E_F$ the set of edges that bound $F$.
Note that if $E_F$ contains a cut-edge, then $|E_F| < |F|$.

We will use the following simple lemma, which essentially follows from `contracting' some curves into points to get a plane graph.

\begin{lem}[{\cite[Lemma 2.6]{keller2020conflictfree}}]\label{lem:planar}
Let $\cS$ and $\cC$ be two sets of curves such that the curves within each set are pairwise disjoint and every curve in $\cC$ intersects exactly two curves from $\cS$.
Then $\cD(\cS,\cC)$ is a planar graph.
\end{lem}

Next we prove Theorem~\ref{thm:non-neighboring} which is equivalent to Theorem~\ref{thm:touchings}. 
Let $\cS$ be a set of $1$-intersecting red and blue curves, such that no two curves of the same color intersect, and let $\cC$ be another set of curves such that each curve in $\cC$ intersects a pair of disjoint curves from $\cS$ and no other curve from $\cS \cup \cC$.
We will show that $|\cC| = O(n)$.
Observe first that it follows from Lemma~\ref{lem:planar} that there are at most $3n-3$ curves in $\cC$ such that each of them intersects two curves from $\cS$ of the same color.\footnote{We sometimes use the weaker bound $3n-3$ (resp., $2n-2$) for the maximum size of an $n$-vertex (resp., bipartite)  planar graph, since it holds for every $n$.}

It remains to bound the number of curves in $\cC$ that intersect curves of different colors.
Next we consider only this set of curves $\cC'$.
Observe that each curve $c \in \cC'$ contains a sub-curve $c'$ whose endpoints are on edges of different colors that bound the same face of $\cA_{\cS}$ and whose interior is contained in that face.
We replace every curve in $\cC'$ with one of its sub-curves with these properties and, by a slight abuse of notation, keep using $\cC'$ to denote the new set of (sub-)curves.

By slightly extending each curve in $\cS$, every intersection point of two curves becomes a proper crossing of them.
Denote by $x$ the number of intersection points of two curves from $\cS$. 
Thus $\cal A_\cS$ has $2n+x$ vertices and $n+2x$ edges.
Let $\cF$ be the face set of $\cal A_\cS$ and let $e$ denote the number of edges of $\cal A_\cS$.
Then by Euler's formula we have $e=n+2x \le (2n+x)+|\cF|-2$, and therefore, $x\le n+|\cF|-2$.

We now add $4x$ dummy curves that connect the four pairs of neighboring edges of $\cal A_\cS$ around every crossing point. These curves are drawn such that they do not intersect the other curves.
For a face $F \in \cF$ let $\cC'_F$ be the curves in $\cC'$ whose interiors are inside $F$.
Then it follows from Lemma~\ref{lem:planar} that $\cD(E_F,\cC'_F)$ is a bipartite planar graph and therefore the number of its edges is at most $2|E_F|-4$ (since a red curve and a blue curve may intersect at most once, $\cA_\cS$ does not contain size-two faces). Since $\sum_{F \in \cF} |E_F| \le 2e$ we have:
$|\cC'|+4x \le \sum_{F \in \cF} (2|E_F|-4) \le 4e-4|\cF| = 4n+8x-4|\cF|.$ 
Hence, $|\cC'| \le 4n-4|\cF|+4x \le 4n-4|\cF|+4n+4|\cF|-8 = 8n-8$, and the number of edges in $\cD(\cS,\cC)$ is at most $11n-11$. 
This finishes the proof.

\paragraph{Remarks.}
(1)~we have only used that curves from $\cC$ do not connect \emph{neighboring} edges of the arrangement of red and blue curves, that is, edges that are consecutive along some face of the arrangement---this is weaker than requiring that each connected pair of red and blue curves involves disjoint curves.

However, if $\cS$ is the union of three sets of pairwise disjoint curves instead of two, then we can no longer claim that $\cD(\cS,\cC)$ has linearly many edges when the curves in $\cC$ may connect only non-neighboring edges (of possibly intersecting curves).
Indeed, consider a triangular grid formed by $n$ segments ($n/3$ segments of each of three directions).
By slightly shifting the segments of one orientation parallel to themselves, each of the $\Omega(n^2)$ intersection points is replaced by a triangle which is adjacent to at least one hexagon.
For each such hexagon it is possible to connect two non-neighboring, non-parallel edges by a curve that intersects no other edges. Thus the number of these curves is $\Omega(n^2)$.

(2)~The only place where we used that each pair of a red curve and a blue curve intersects at most once is for claiming that $\cA_{\cS}$ does not contain size-two faces. Therefore Theorem~\ref{thm:non-neighboring} remains true even when each such pair intersects finitely many times, as long as there are no size-two faces in $\cA_{\cS}$.

\section{Improving Lemma~\ref{lem:keller}}
\label{sec:linear}

In this section we improve the bound in Lemma~\ref{lem:keller} by proving the following weak version of Theorem~\ref{thm:groundedcurves}. 
This bound suffices for obtaining the upper bound in Theorem~\ref{thm:cf-grounded}.

\begin{lem}\label{lem:keller-better}
	Let $\cS$ be a set of $n$ axis-parallel line-segments and let $\cC$ be a set of pairwise disjoint curves grounded with respect to $\cS$.
	Then $\cD(\cS,\cC)$ has at most $13n-11$ edges. 
\end{lem}

By a slight perturbation of the segments if needed, we may assume that parallel segments do not fall on the same line---this does not decrease the number of edges of $\cD(\cS,\cC)$, so from now on we assume this is the case.

Recall that $E_F$ denotes the set of edges that bound a face $F$ in the arrangement $\cA_\cS$, where cut-edges are counted once.
Next we would like to bound $|E_F|$ where $\cS$ is a set of axis-parallel segments.
The following fact is probably known, however, we provide a proof for completeness and since we could not find any reference to it.

\begin{lem}\label{lem:segments-cell-complexity}
	Let $\cS$ be a set of $n$ axis-parallel line-segments, $n>1$, and let $F$ be a face of $\cal A_S$.
	Then $|E_F| \le 4n-4$.
	This bound is tight.
\end{lem}

\begin{proof}
	By slightly extending every segment we may assume that if two segments intersect, then they intersect in their interiors. This step cannot decrease $|E_F|$.
	
	The boundary of $F$ consists of at most one outer closed walk $W_0$ (which surrounds $F$) and possibly also some inner closed walks $W_1,W_2,\ldots,W_k$ (which are surrounded by $F$). 
	Denote by $x_i$ the number of appearances of intersection points of two segments that one encounters while going along $W_i$, and by $y_i$ the number of segment-endpoints along $W_i$.
	Note that an intersection point might be counted with multiplicity, whereas every endpoint is counted exactly once.
	Let $e_i$ denote the number of edges along $W_i$, such that each edge is counted with multiplicity one, so
	$e_i\le |W_i|$.
	By definition $|E_F|=\sum_{i=0}^k e_i$.
	
	Clearly, if $|W_i|=2$, i.e., $W_i$ consists of a single edge, then $e_i=1 \le 4 = 2y_i$.
	Otherwise, $e_i \le x_i$. 
	Indeed, associate every edge along $W_i$ with its preceding vertex (thus, a cut-edge is associated with two vertices).
	If the preceding vertex of an edge $e$ is a segment-endpoint, then the vertex before that endpoint is an intersection point of two segments which is also associated with $e$.
	Therefore, every edge is associated with an intersection point and every intersection point is associated to as many edges as its number of appearances along the walk to which it belongs.
		
	Suppose that we traverse an inner walk $W_i$ such that $F$ is to our left.
	Then at each endpoint we turn $180^\circ$ in the positive direction and at each intersection point we turn $90^\circ$ in the negative direction.
	As the sum of positive and negative turns should be $360^\circ$, we have that $x_i=2y_i-4$, thus $e_i\le 2y_i-4$.
	If $W_0$ exists, then a similar calculation gives $x_0=2y_0+4$, thus $e_0\le 2y_0+4$.
	
	Let $y = \sum_{i=0}^k y_i$ be the number of segment-endpoints on the boundary of $F$.
	If there is no outer walk, then either there are no intersection points, in which case $|E_F| \le n \le 4n-4$, since $n>1$. Otherwise, for at least one index $j$ we have $|W_j|>2$. In this case $|E_F|=\sum_{i=0}^k e_i \le 2y_j-4 + \sum_{i \ne j} 2y_i \le 2y-4$. As $y \le 2n$ we get that $|E_F| \le 4n-4$.
	
	If there is an outer walk,	then we can only conclude that $|E_F| \le 2y+4$. However, in this case the topmost and bottommost horizontal segments and the leftmost and rightmost vertical segments do not contribute any endpoints to $y$, thus $y \le 2n-8$ and therefore $|E_F| \le 4n-12$. 
	
	\medskip
	It is easy to see that this bound is tight by considering the outer face in an arrangement of $h$ ($1 < h < n-1$) horizontal segments and $n-h$ vertical segments in a grid-like arrangement.
\end{proof}

The last ingredient needed for proving Lemma~\ref{lem:keller-better} is the following. 

\begin{lem}\label{lem:3-curve-sets}
	Let $\cS$ be a set of red and blue curves, such that any two curves of the same color do not intersect.
	Suppose that $\cC$ is a set of pairwise disjoint curves such that each curve in $\cC$ intersects exactly a distinct pair of curves from $\cS$ (each of them possibly several times) and there is a face $F$ of the arrangement $\cA_{\cS}$ that is intersected by every curve in $\cC$.
	Then $|\cC| \le 5|\cS| + 2|F|-3$.
\end{lem}

\begin{proof}
	It follows from Lemma~\ref{lem:planar} that there are at most $3|\cS|-3$ curves in $\cC$ such that each of them intersects two curves of the same color.
	
	Denote by $E_r$ the red edges of $F$ and let $\cC_{rb} \subseteq \cC$ be the curves that intersect a red edge of $F$ and a blue curve.
	First we will bound $|\cC_{rb}|$.
	Remove every red edge of $\cA_{\cS}$ which is not in $E_r$.
	Note that this might break a red curve into several pieces.
	Now we eliminate the possible multiple crossings of $\cC_{rb}$ with members of $\cS$.
	Every curve in $\cC_{rb}$ has a sub-curve such that one of its endpoints is on an edge from $E_r$, its other endpoint is on a blue curve and its interior does not intersect any other curve or edge.
	Pick such a sub-curve for every curve in $\cC_{rb}$ and denote this set of (sub-)curves by $\cC'_{rb}$.
	
	Clearly $\cD(\cS, \cC_{rb})$ has as many edges as $\cD(E_r \cup \cS_b, \cC'_{rb})$, where $\cS_b \subseteq \cS$ 
	 denotes the set of blue 
	curves.
	Observe also that it follows from Lemma~\ref{lem:planar} that $\cD(E_r \cup \cS_b, \cC'_{rb})$ is a bipartite planar graph.
	Thus, $\cD(E_r \cup \cS_b, \cC'_{rb})$ has at most $2(|E_r|+|\cS_b|)$ edges.
	
	By applying the same argument for the similarly defined curves $\cC_{br} \subseteq \cC$ that intersect a blue edge of $F$ and a red curve, we conclude that $|\cC| \le 2(|E_r|+|\cS_b|)+2(|E_b|+|\cS_r|) + 3|\cS|-3 = 5|\cS|+2|F|-3$.
\end{proof}

Lemma~\ref{lem:keller-better} now follows from Lemma~\ref{lem:segments-cell-complexity} and Lemma~\ref{lem:3-curve-sets}. From these we get that the number of edges in  $\cD(\cS,\cC)$ is upper bounded by $5|\cS|+2|F|\le 5|\cS|+2(4|\cS|-4)=13|\cS|-11$.

\subsection{An application for conflict-free coloring of L-shapes}
\label{sec:cf}

In this section we outline the proof of the upper bound of Theorem~\ref{thm:keller} from~\cite{keller2020conflictfree}, and indicate how using Theorem~\ref{thm:groundedcurves} (or just Lemma~\ref{lem:keller-better}) improves on this by a $\log n$ factor, proving Theorem~\ref{thm:cf-grounded}.

The proof in~\cite{keller2020conflictfree} uses a divide-and-conquer approach.
The family of $n$ grounded L-shapes is partitioned with respect to some vertical line into three sets, $\cL_1$, $\cL_2$ and $\cL_3$, such that: $|\cL_1|, |\cL_3| \le n/2$; no L-shape from $\cL_1$ intersects an L-shape from $\cL_3$; each L-shape in $\cL_i$ intersects some L-shape in $\cL_i$, for $i=1,2,3$; and $\chi_{\rm CF}(\cH(\cL_2)) = O(\log^2 n)$.
By using the same set of colors for $\cL_1$ and $\cL_3$, and applying induction the $O(\log^3 n)$ upper bound on $\chi_{\rm CF}(\cH(\cL))$ follows.

In order to show that $\chi_{\rm CF}(\cH(\cL_2)) = O(\log^2 n)$, Keller et al.~\cite{keller2020conflictfree} show that $\chi(\cH(\cL'_2)) = O(\log |\cL'_2|)$
for any $\cL'_2 \subseteq \cL_2$ and rely on the following known result.

\begin{lem}\label{lem:weak-coloring}\cite{Har-PeledS05,keller2020conflictfree}
	If an $n$-vertex hypergraph $\cH$ as well as any induced sub-hypergraph of $\cH$ can be properly colored with $t$ colors, then $\chi_{\rm CF}(\cH) = O(t\log n)$.
\end{lem}

The proper coloring of $\cH(\cL_2)$ (and its induced sub-hypergraphs) is obtained by representing $\cL_2$ as the union of four subsets, with their notation $\cL_2=\cI \cup \cF_M \cup \cV'_R \cup \cV'_L$,
where $\cI$ is a set of pairwise disjoint L-shapes such that each of them intersects exactly two other L-shapes.
Then it is basically shown that $\chi(\cH(\cL_2, \cF_M \cup \cV'_R \cup \cV'_L)) = O(1)$ and that $\chi(\cH(\cL_2, \cI)) = \chi(\cD(\cL_2, \cI)) = O(\log n)$, thus concluding that $\chi(\cH(\cL_2)) = O(\log n)$.
The bound $\chi(\cD(\cL_2, \cI)) = O(\log n)$ follows from the $O(n \log n)$ bound on the number of edges in $\cD(\cL_2, \cI)$ as stated in Lemma~\ref{lem:keller}.
By substituting this bound with our linear upper bound that follows from Lemma~\ref{lem:keller-better} or Theorem~\ref{thm:groundedcurves}, we get that $\chi(\cD(\cL_2, \cI))$ is upper bounded by a constant and therefore so is $\chi(\cH(\cL_2))$.
Then by Lemma~\ref{lem:weak-coloring} we have $\chi_{\rm CF}(\cH(\cL_2)) = O(\log n)$ and the bound $\chi_{\rm CF}(\cH(\cL)) = O(\log^2 n)$ follows by induction as before.

\section{Proof of Theorem~\ref{thm:groundedcurves}}
\label{sec:groundedcurves}

In this section we prove Theorem~\ref{thm:groundedcurves}.
We will use the following bound on \emph{Davenport-Schinzel sequences}.
\begin{lem}[\cite{DS-book}]
	\label{lem:DS}
	Let $S$ be a sequence that consists of $n$ symbols, such that no two consecutive symbols are the same and $S$ does not contain a sub-sequence of the form $a,b,a,b$.
	Then the length of $S$ is at most $2n-1$.
\end{lem}

We will also need the following fact.

\begin{prop}\label{prop:diff-consecutive}
Let $S_1=a_1,\ldots,a_n$ and $S_2=b_1,\ldots,b_n$ be two sequences, such that there is no index $j$ for which $a_j=a_{j+1}$ and $b_j=b_{j+1}$ (that is, if $a_{j}=a_{j+1}$, then $b_j \ne b_{j+1}$ and vice versa).
Then there is $i \in \{1,2\}$ such that $S_i$ contains a sub-sequence of length at least $\lceil (n+1)/2 \rceil$ in which every element is different from its preceding element in the sub-sequence.  
\end{prop}

\begin{proof}	
We construct two sub-sequences $S'_1$ and $S'_2$ in an incremental and greedy manner.
First, $S'_1$ contains only $a_1$ and $S'_2$ contains only $b_1$.
Then, for every $j>1$ we append $a_j$ (resp., $b_j$) to the sub-sequence $S'_1$ (resp., $S'_2$) if it is different from the last element that was added to the sub-sequence.
Clearly, for every $j>1$ at least one sub-sequence is extended, otherwise there is $j$ such that $a_j=a_{j+1}$ and $b_j=b_{j+1}$.
Therefore, the total length of $S'_1$ and $S'_2$ is at least $2+(n-1)$ and thus
one of them is of length at least $\lceil (n+1)/2 \rceil$.
\end{proof}

Let $\cS$ be a set of $n$ red and blue curves such that no two curves of the same color intersect.
Suppose that $\cC$ is a set of pairwise disjoint curves grounded with respect to $\cS$, such that each of them intersects exactly a distinct pair of curves from $\cS$.
Recall that we wish to show that $|\cC| = O(n)$.
Note that we cannot use Lemma~\ref{lem:3-curve-sets} since the size of a face in $\cA_{\cS}$ can be arbitrarily large.

Observe first that it follows from Lemma~\ref{lem:planar} that there are at most $3n-3$ curves in $\cC$ such that each of them intersects two curves in $\cS$ of the same color.
We thus discard such curves from $\cC$ and assume henceforth that every curve in $\cC$ connects two curves from $\cS$ of different colors.

Let $F$ be a face of $\cA_{\cS}$ such that every curve in $\cC$ intersects its boundary.
By trimming every curve in $\cC$ if necessary, we may assume that each such curve has one of its endpoints on an edge of the boundary of $F$ that belongs to a red (resp., blue) curve $s \in \cS$, its other endpoint is on a blue (resp., red) curve $s' \in \cS$, and its interior does not intersect any curve in $\cC \cup \cS$ except for possibly $s$ and $s'$.

Let $\cC_1 \subseteq \cC$ be the curves with exactly one endpoint on the boundary of $F$ and let $\cC_2 = \cC \setminus \cC_1$ be the curves with two endpoints on the boundary of $F$.
Note that we may assume that the interior of every curve $c_1 \in \cC_1$ intersects at most one curve, namely the curve that contains that edge of $F$ that $c_1$ intersects. 
Furthermore, we may assume that the interior of every curve $c_2 \in \cC_2$ does not intersect any curve (including $s$ and $s'$, as defined above). Indeed, otherwise $c_2$ contains a sub-curve that qualifies for $\cC_1$ and we may replace $c_2$ with this sub-curve (see Figure~\ref{fig:C_1C_2} for an illustration).
\begin{figure}
\begin{center}
	\includegraphics[width= 8cm]{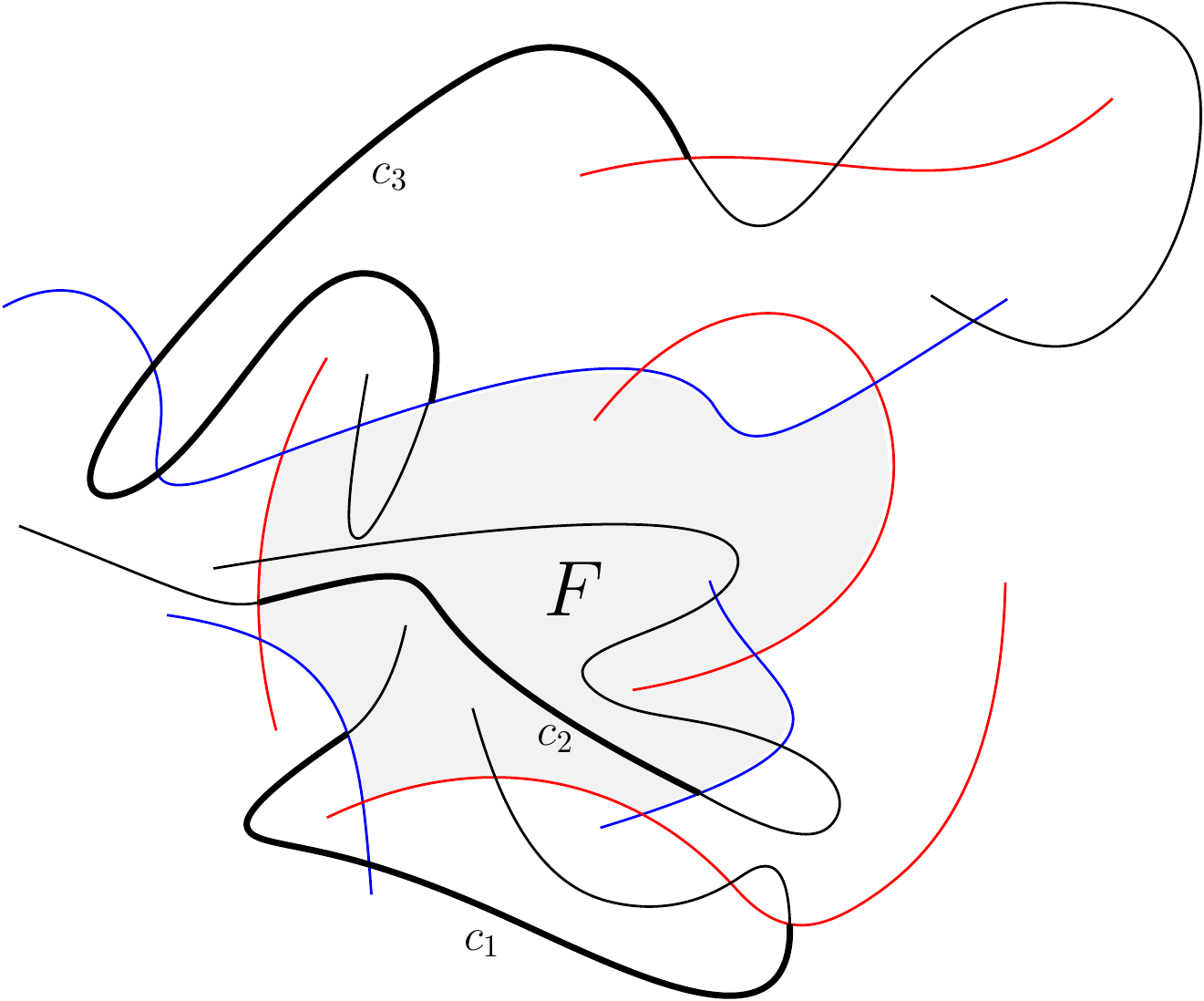}
	\caption{The sub-curves $c_1$ and $c_3$ belong to $\cC_1$ whereas $c_2 \in \cC_2$.}
	\label{fig:C_1C_2}
\end{center}
\end{figure}

\begin{lem}\label{lem:C_1}
	$|\cC_1| \le 8n-4$.
\end{lem}

\begin{proof}
	We can assume that the interior of each curve in $\cC_1$ is disjoint from the interior of $F$, otherwise we could take one of its sub-curves.
The boundary of $F$ may consist of several connected components. 
Let $s$ be a curve in $\cS$.
Clearly, it is impossible that $s$ contributes edges to two different components.
Furthermore, it is also impossible that there are curves $c,c' \in \cC_1$ that connect $s$ to two points on different connected components of $F$, since one of $c$ and $c'$ must then intersect the boundary of $F$ at two points.   
Therefore, it is enough to consider one connected component.
That is, let $\cC'_1 \subseteq \cC_1$ be the curves with endpoints on a specific connected component of the boundary of $F$ and let $\cS' \subseteq \cS$ be the curves that intersect at least one curve from $\cC'_1$.
Then it is enough to show that $|\cC'_1| = O(|\cS'|)$, since by summing for every connected component of the boundary of $F$ we get $|\cC_1| = O(|\cS|)$.

Let $\cC''_1 \subseteq \cC'_1$ be the curves whose endpoints on $F$ belong to red curves. We may assume without loss of generality that $|\cC''_1| \ge |\cC'_1|/2$.
Let $p_1,p_2,\ldots,p_k$ be the endpoints of the curves in $\cC''_1$ on $F$ listed in their cyclic order along the boundary of $F$.
For $i=1,2,\ldots,k$, let $c_i \in \cC''_i$ be the curve whose endpoint is $p_i$ and let $r_i$ and $b_i$ be the red and blue curves, respectively, that are connected by $c_i$.
Observe that it is possible that $b_i=b_j$ or $r_i=r_j$ for $i \ne j$. 
However, it is clearly impossible that for some $i$ we have $b_i=b_{i+1}$ and $r_{i}=r_{i+1}$, for then $c_i$ and $c_{i+1}$ connect the same pair of curves.
Therefore, it follows from Proposition~\ref{prop:diff-consecutive} that one of sequences $S_1=b_1,b_2,\ldots,b_k$ and $S_2=r_1,r_2,\ldots,r_k$ contains a sub-sequence of length at least $\lceil (k+1)/2 \rceil \ge k/2$.

\begin{prop}
	$S_1$ does not contain a sub-sequence $b_i,b_j,b_{i'},b_{j'}$ such that $b_i=b_{i'}\ne b_j=b_{j'}$.
\end{prop}

\begin{proof}
Suppose for contradiction that such a sub-sequence exists.
Let $\alpha_i$ be the 
curve that consists of $c_i, c_{i'}$ and the sub-curve of $b_i$ between the endpoints of $c_i$ and $c_{i'}$ on $b_i$.
Define $\alpha_j$ analogously. 
Consider a closed curve $c'$ that consists of $\alpha_i$ and a curve within $F$ that connects $p_i$ and $p_{i'}$.
Due to the alternating sub-sequence of symbols, the points $p_{j}$ and $p_{j'}$ lie on different sides of $c'$.
It follows that $\alpha_j$, which connects these two points, must cross $c'$, which is impossible since $b_i$ and $b_j$ do not intersect and the interiors of the curves in $\cC$ are intersection-free.
\end{proof}

\begin{prop}\label{prop:no-ABAB-on-boundary}
	$S_2$ does not contain a sub-sequence $r_i,r_j,r_{i'},r_{j'}$ such that $r_i=r_{i'}\ne r_j=r_{j'}$.
\end{prop}

\begin{proof}
	Suppose for contradiction that such a sub-sequence exists.
	Consider a closed curve $c'$ that consists of the subcurve of $r_i$ between $p_i$ and $p_{i'}$ and a curve within $F$ that connects $p_i$ and $p_{i'}$.
	Due to the alternating sub-sequence of symbols, the points $p_{j}$ and $p_{j'}$ lie on different sides of $c'$.
	It follows that $r_j$, which connects these two points, must cross $c'$, which is impossible.
\end{proof}

It follows from Lemma~\ref{lem:DS} that $k/2 \le 2|\cS'|-1$ and hence $|\cC''_1| \le 4|\cS'|-2$ and $|\cC_1| \le 8n-4$.
\end{proof}

\begin{lem}\label{lem:C_2}
	$|\cC_2| \le 22n-18$.
\end{lem}

\begin{proof}
Recall that the boundary of $F$ may consist of several connected components.
We consider first the curves $\cC'_2 \subseteq \cC_2$ which connect edges of $\cA_{\cS}$ that belong to different connected components of the boundary of $F$.

For each connected component choose either `red' or `blue' uniformly at random.
If `red' (resp., `blue') is chosen for a certain connected component, then we delete all the red (resp., blue) curves that contain edges of this connected component.
Every curve in $\cC'_2$ survives (with probability $1/4$) if both curves that contain its endpoints survive.
Since red and blue curves from different components do not intersect, it follows from Lemma~\ref{lem:planar} that the surviving curves define a bipartite plane graph whose vertices correspond to surviving red and blue curves and its edges correspond to surviving curves in $\cC'_2$.
The expected number of vertices of this graph is at most $n/2$ and the expected number of edges of this graph is $|\cC'_2|/4$.
Thus we have $|\cC'_2|/4 \le 2\cdot n/2$ and hence  $|\cC'_2| \le 4n$.

It remains to bound the number of curves in $\cC_2$ that connect edges of $\cA_{\cS}$ that belong to the same connected component of the boundary of $F$.
Let $\cC''_2 \subseteq \cC_2$ be such curves for a certain connected component and let $\cS'' \subseteq \cS$ be the set of curves that contain edges of this component.
Note that it is enough to show that $|\cC''_2| = O(|\cS''|)$, since then by summing for every connected component of the boundary of $F$ we get $|\cC_2| = O(|\cS|)$.

Denote the edges of the certain connected component of the boundary of $F$ that we consider by $e_1,e_2,\ldots,e_{k}$, and let $s_1,s_2,\ldots,s_k$ be the corresponding curves (with possible repetitions).
We partition $e_1,\ldots,e_{k}$ into \emph{alternating runs}:
The first alternating run is the maximal sequence of edges $e_{1},\ldots,e_i$ that belong to exactly two curves (necessarily a blue curve and a red curve).
The next alternating run is the maximal sequence of edges $e_{i+1},e_{i+2},\ldots,e_j$ that belong to exactly two curves, and so on and so forth.
Let $m$ denote the number of alternating runs.

\begin{prop}
$m \le 2|\cS''|-1$.
\end{prop}

\begin{proof}
Let $S_1$ be the sequence we get by starting with $s_1$ and adding $s_{2i+1}$ if it is different from $s_{2i-1}$, for $i=1,2,\ldots,\lfloor k/2 \rfloor$.
Similarly, let $S_2$ be the sequence we get by starting with $s_2$ and adding $s_{2i}$ if it is different from $s_{2i-2}$, for $i=1,2,\ldots,\lfloor k/2 \rfloor$.
By definition, every element in these sequences is different from its preceding element. 
Furthermore, after each alternating run an element is added to one of $S_1$ and $S_2$, thus their total length is at least $m$.
On the other hand, as in the proof of Proposition~\ref{prop:no-ABAB-on-boundary} we may conclude that none of $S_1$ and $S_2$ contains a sub-sequence of the form $a,b,a,b$.
Therefore, by Lemma~\ref{lem:DS} the total length of $S_1$ and $S_2$ is at most $2|\cS''|-1$. Thus, $m \le 2|\cS''|-1$.
\end{proof}

Clearly, $\cC''_2$ contains at most $m$ curves that connect two edges which belong to the same alternating run. 
Next, we bound the number of remaining curves $\cC'''_2$, that is, those connecting edges from different alternating runs.
To this end, for each alternating run choose either `red' or `blue' uniformly at random.
If `red' (resp., `blue') is chosen for a certain alternating run, then we delete all red (resp., blue) edges of that run and the curves in $\cC'''_2$ that have an endpoint on one of them.
Thus, every curve in $\cC'''_2$ survives with probability $1/4$.

Consider the graph such that each of its vertices is the union of the remaining edges of a certain alternating run and whose edges correspond to surviving curves in $\cC'''_2$.
By Lemma~\ref{lem:planar} it is a planar (bipartite) graph.
Since the expected number of surviving curves in $\cC'''_2$ is $|\cC'''_2|/4$, we conclude that $|\cC''_2| \le 4\cdot 2m + m = 9m \le 18|\cS''|-18$.

By summing for every connected component of the boundary of $F$ and recalling that $|\cC'_2| \le 4n$, we have $|\cC_2| \le 22n-18$.
\end{proof}

From Lemma~\ref{lem:C_1} and Lemma~\ref{lem:C_2} we have $|\cC| \le 3n-3+8n-4+22n-18 = 33n-25$.
This concludes the proof of Theorem~\ref{thm:groundedcurves}.

\section{Bounding the number of hyperedges}\label{sec:general-bounds}

In this section we recall how linear upper bounds on the size of  Delaunay-graphs of induced sub-hypergraphs of a hypergraph $\cH$ imply upper bounds on its chromatic number, VC-dimension and number of hyperedges of size at most $k$.

\begin{proof}[Proof of Theorem \ref{thm:2unif-kunif}~(i)]
	Let $\cH$ be a hypergraph as assumed. Then the average degree of the Delaunay-graph of every induced sub-hypergraph of $\cH$ is at most $2c_1$ (and strictly less than $2c_1$ if $c_2>0$) and thus it has a vertex of that degree or smaller. 
	Now we can easily get a proper $(2c_1+1)$-coloring (even a $2c_1$-coloring if $c_2>0$) by induction. 
	Remove a vertex $v$ with the smallest degree in the Delaunay-graph of $\cH$ and let $\cH'$ be the hypergraph which is induced by the remaining vertices. As $\cH'$ still has the above property, by induction it has a proper $(2c_1+1)$-coloring (even a $2c_1$-coloring if $c_2>0$). Now color $v$ with a color that is different from the colors of all of its neighbors in the Delaunay-graph of $\cH$. We claim that this is a proper coloring of $\cH$. Indeed, if a hyperedge contains at least two vertices other than $v$, then it is non-monochromatic by induction, otherwise it contains exactly two vertices, one of them being $v$ and then it is non-monochromatic by the choice of color for $v$.
\end{proof}

\begin{proof}[Proof of Theorem \ref{thm:2unif-kunif}~(ii)]
	Suppose that that the VC-dimesion of $\cH=(\cV,\cE)$ is $d$. 
	Then there exists a subset $\cV' \subseteq \cV$ of size $d$ such that it induces a hypergraph containing all the subsets of $\cV'$. In particular, it contains the $d\choose 2$ hyperedges of size two, and thus ${d\choose 2}\le c_1d-c_2$, by the properties of $\cH$. Thus $d\le \frac{(2c_1+1)+\sqrt{(2c_1+1)^2-8c_2}}{2}$. Since $d$ is an integer, we have $d\le 2c_1+1$ and $d\le 2c_1$ if $c_2>0$.
\end{proof}

The proof of part (iii) is more involved, but it does not require new ideas; we follow standard techniques and the proofs in~\cite{aronov21} and~\cite{buzaglo} (almost verbatim).

For a hypergraph $\cH=(\cV,\cE)$ and an integer $k \ge 2$ we denote by $\cE_k$ (resp., $\cE_{\le k}$) the set of hyperedges whose size is exactly (resp., at most) $k$.
We say that $\cH$ is \emph{$c$-linear} for some constant $c$, if $|\cE'_2| \le c|\cV'|$ for every induced sub-hypergraph $(\cV',\cE')$.
We first consider an upper bound for $|\cE_{\le k}|$, following the methods in~\cite{aronov21,buzaglo}.

For $k \ge 2$, a pair of vertices of a hypergraph $\cH$ is \emph{$k$-good} if there exists a hyperedge of size at most $k$ in $\cH$ which contains both vertices.

\begin{lem}\label{lem:k-good}
	Let $\cH$ be an $n$-vertex $c$-linear hypergraph for some constant $c$. Then for every $k \ge 2$ there are at most $neck$ $k$-good pairs in $\cH$.\footnote{The $e$ in $neck$ stands for Euler's number.}
\end{lem}

\begin{proof}
Set $q=1/k$ and remove every vertex of $\cH$ independently with probability $1-q$ to get a vertex set $\cV'$ which induces a sub-hypergraph $\cH'=(\cV',\cE')$. A $k$-good pair of $\cH$ becomes a size-two hyperedge of $\cH'$ with probability at least $q^2(1-q)^{k-2}$. On the other hand $|\cE'_2| \le c|\cV'|$ since $\cH$ is $c$-linear. Let $g$ denote the number of $k$-good pairs of $\cH$. Then $gq^2 (1-q)^{k-2}\le \mathbb{E}[|\cE'_2|] \le c\mathbb{E}[|\cV'|]=cqn$.
	Thus, we get $g\le \frac{cn}{\frac1k(1-\frac1k)^{k-2}}\le neck$, as required.
\end{proof}

\begin{lem}[\cite{buzaglo}]\label{lem:Kn}
	Let $G$ be an $n$-vertex $c$-linear graph for some constant $c$.
	Then, for every $h \ge 2$, the number of copies of $K_h$ (the complete graph on $h$ vertices) in $G$ is at most $t_{c,h}\,n$, where $t_{c,h} = \frac{ (2c)^{h-1}}{h!}$.
\end{lem}

\begin{cor}\label{cor:bounded}
	If $\cH=(\cV,\cE)$ is a $c$-linear hypergraph, then $|\cE_{\le k}| \le b_{c,k}\, n$, where $b_{c,k}$ depends only on $c$ and $k$.
\end{cor}

\begin{proof}
	Define a graph $G$ whose vertex set is $\cV$ such that there is an edge between each pair of $k$-good vertices. It follows from Lemma \ref{lem:k-good} that $G$ is $(eck)$-linear. 
	By Lemma~\ref{lem:Kn} for every $2 \le h \le k$ the number of copies of $K_h$ in $G$ is at most $\frac{(2eck)^{h-1}}{h!}\,n$.
	Therefore $|\cE_{\le k}| \le b_{c,k}\,n$ where $b_{c,k} = 1+ \sum_{h=2}^k \frac{(2eck)^{h-1}}{h!}$.
\end{proof}

Note that the constant $b_{c,k}$ in Corollary~\ref{cor:bounded} is huge.
Next, we will use this bound to obtain an upper bound of the form $|\cE_{\le k}| \le O_{c,d}(k^{d-1}n)$, where $d$ is the VC-dimension of $\cH$ (often quite small).

The following is a well-known property of hypergraphs of bounded VC-dimensions, a stronger form of the Sauer-Shelah-lemma, used also in Buzaglo et al.~\cite{buzaglo}.

\begin{thm}\label{thm:vcinj}
	Let $\cH$ be a hypergraph with VC-dimension $d$. Then it is possible to assign to each hyperedge a subset of at most $d$ of its vertices, such that distinct hyperedges are assigned distinct subsets.
\end{thm}

\begin{proof}[Proof of Theorem \ref{thm:2unif-kunif}~(iii)]
Let $\cH=(\cV,\cE)$ be an $n$-vertex $c$-linear hypergraph with VC-dimension at most $d$ and let $k \ge 2$ be an integer.
We wish to show that $|\cE_{\le k}| \le O_{c}(k^{d-1}n)$ using an argument similar to the proof of Lemma \ref{lem:k-good}.

By Theorem \ref{thm:vcinj} we can assign to every hyperedge $h$ of $\cH$ a signature $h'\subseteq h$ of size at most $d$.
Set $q=1/k$ and remove every vertex of $\cH$ independently with probability $1-q$ to get a vertex set $\cV'$ that induces a sub-hypergraph $\cH'=(\cV',\cE')$. We say that a hyperedge $h \in \cE_{\le k}$ \emph{survives} if $h\cap \cV'=h'$, where $h'$ is the signature of $h$.
Observe that a hyperedge $h$ (of size at most $k$) survives with probability 
$$q^{|h'|}(1-q)^{|h|-|h'|}\ge q^{|h'|}(1-q)^{k-|h'|}\ge q^{d}(1-q)^{k-d},$$
where the first inequality holds since $|h| \le k$ and the second inequality holds since $q \le 1-q$ and $|h'| \le d$.

By Corollary~\ref{cor:bounded} we have $|\cE'_{\le d}| \le b_{c,d}|\cV'|$, where $b_{c,d}$ is the constant from the corollary. 
Therefore, 
$$q^{d}(1-q)^{k-d} |\cE_{\le k}| \le \mathbb{E}[|\cE'_{\le d}|] \le b_{c,d}\, \mathbb{E}[|\cV'|]=b_{c,d}qn.$$
This implies that $|\cE_{\le k}| \le b_{c,d}(1-q)^{d-k}q^{1-d}n \le b_{c,d}\,k^{d-1}\,n$.
Since $d \le 2c+1$ by Theorem \ref{thm:2unif-kunif}~(ii) we have that the number of hyperedges of size at most $k$ is $O(k^{d-1}n)$ where the constant hiding in the big-$O$ notation depends only on $c$.
\end{proof}

\paragraph{Remarks.}
Theorem~\ref{thm:groundedcurves} and Theorem~\ref{thm:2unif-kunif} together imply Corollary~\ref{cor:curves-groundedcurves}. Observe also that Lemma~\ref{lem:keller-better} and Theorem~\ref{thm:2unif-kunif} imply that there exists a proper $26$-coloring of $\cH(\cS,\cC)$,
for every family $\cS$ of $n$ axis-parallel segments and a family $\cC$ of pairwise disjoint curves.
Furthermore, this hypergraph has VC-dimension at most $26$ and has $(k^{25}n)$ hyperedges of size at most $k$.

However, it is easy to get better upper bounds on the chromatic number and the VC-dimension of these hypergraphs, even when $\cS$ consists of red and blue curves instead of segments.
Indeed, the VC-dimension is at most $8$, since a shattered set of $9$ vertices would imply $5=\lceil \frac92\rceil$ shattered curves of the same color, which would give a plane drawing of $K_{5}$. 
Furthermore, recall that the Delaunay-graph of the red (resp., blue) curves with respect to the curves in $\cC$ is planar, and therefore $8$ colors suffice for coloring the red and blue curves such that the corresponding hypergraph is properly colored (four different colors are used for each of the two).

\subsubsection*{Acknowledgement}

We thank Chaya Keller for pointing out that Theorem \ref{thm:2unif-kunif} also implies Theorem \ref{thm:hittingsetfromdel}.

\footnotesize
\bibliographystyle{plainurl}
\bibliography{psdisk}

\end{document}